\documentclass[preprint,3p]{elsarticle}

\usepackage{amssymb}
\usepackage{amsmath,amsthm,amscd}
\usepackage{amsfonts,mathrsfs}
\usepackage{graphicx,color,caption}
\usepackage{float,setspace,subfigure}

\usepackage{enumerate,amsmath,amsthm,amsfonts}
\usepackage{amsmath}
\usepackage{amsfonts, relsize}
\usepackage{amssymb, accents, xcolor}
\usepackage{epsfig}
\usepackage{graphicx,color}
\usepackage{graphpap,color}
\usepackage{graphicx}
\usepackage{latexsym}
\usepackage{float}
\usepackage{setspace}

\newtheorem{theorem}{Theorem}[section]
\newdefinition{definition}[theorem]{Definition}
\newtheorem{corollary}[theorem]{Corollary}
\newtheorem{proposition}[theorem]{Proposition}
\newtheorem{conjecture}[theorem]{Conjecture}
\newtheorem{problem}[theorem]{Problem}

\newdefinition{example}[theorem]{Example}
\newdefinition{remark}[theorem]{Remark}
\numberwithin{equation}{section}

\begin{document}

\begin{frontmatter}



\title{Quasi-($\lambda,n $)-distance-balanced graphs
}


\cortext[footnoteinfo]{Corresponding author}

 \author[LNU,IUST]{Ehsan Pourhadi\corref{footnoteinfo}}
\ead{epourhadi@alumni.iust.ac.ir}

\author[PNU]{Morteza Faghani }
\ead{m$_{-}$faghani@pnu.ac.ir}

\address[LNU]{International Center for Mathematical Modelling in Physics and Cognitive Sciences MSI, Linnaeus University, SE-351 95, V\"{a}xj\"{o}, Sweden}
 \address[IUST]{Department of Mathematics, Iran University of Science and Technology, Narmak, Tehran 16846-13114, Iran}

\address[PNU]{Department of Mathematics, Payame Noor University, P.O.Box 19395-3697, Tehran, Iran}


\begin{abstract}
	For every pair of vertices $u$ and $v$ with $ d(u,v)=n $, $ W_{u\underline{n}v}^{G} $ 
	denotes the set of all vertices of $ G $ that are closer to $ u $ than to $ v $.
		A graph $G$ is said to be quasi-($\lambda,n $)-distance-balanced if $ | W_{u\underline{n}v}^{G}|=\lambda^{\pm1}|W_{v\underline{n}u}^{G} | $, for some positive 
	rational number $ \lambda > 1. $ In this paper, we study some properties of these graphs and present a formula to construct such graphs for arbitrarily diameter $ d $.
	For $ n=1 $, this class of graphs contains the quasi-$\lambda$-DB graphs recently introduced by Abedi et al.  [Quasi-$\lambda $-distance-balanced graphs, 
	Discrete Appl. Math. 227 (2017) 21--28]. 
Moreover, we will take a look at the problems arisen by  Abedi et al. Some problems and a conjecture are involved. 
\end{abstract}

\begin{keyword}
Distance-balanced graph \sep quasi-$\lambda $-distance-balanced graph \sep bipartite graph.

\MSC[2010] 05C12

\end{keyword}

\end{frontmatter}


\section{Introduction}
\label{S1}

It is well-known that in graph theory, the distance-balanced graphs are considered as 
one of the important class of graphs (see \cite{R3},\cite{R4},\cite{p},\cite{h},\cite{c}-\cite{n} and the references therein). The significance of these graphs is evident from
their applications in various areas, especially theoretical computer science (more precisely, balance in communication networks), 
and molecular analysis in chemical studies. 
One of the motivations of distance-balanced property is its application in partitioning the network topology into two equal pieces of nodes, the halves may have a very different structure, in particular their metric properties can be very different. If we have an option to design a network in advance (say, in the situation when two parties are competing in a common market with an objective to minimize the cost of transport between all its nodes, it seems fair to design a network in such a way that neither of the involved parties has an advantage to the other). In another word, structuring the distance-balanced graphs brings us the fairness in distribution of benefits.  
\\

Let $G$ be a finite, undirected and connected graph with diameter $d$, and let $V(G)$ and $E(G)$ indicate the vertex set and the edge set of $G$, respectively. For $a,b\in V(G)$, let $d_{G}(a,b)$ (or simply $ d(a,b) $) stand for the minimal path-length distance between $a$ and $b$. For any pair of vertices $a,b$ of $G$ with $ d(a,b)=n $, we denote
\begin{eqnarray*}
	W_{a\underline{n}b}^{G}=\{x\in V(G) \hskip 1mm | \hskip 1mm d(x,a)<d(x,b)\},
\end{eqnarray*}
and 
\begin{eqnarray*}
\underaccent{\mathlarger{a} \ \underline{\mathlarger{n}} \ \mathlarger{b}}{W}^{G}=\{x\in V(G)\hskip 1mm | \hskip 1mm d(x,a)=d(x,b)\}.
\end{eqnarray*}
In 2017, Abedi et al. \cite{R1} presented a class of graphs, so-called \textit{quasi-$\lambda$-distance-balanced} (DB) graphs, 
in which either
	$ \vert W_{ab}^{G}\vert=\lambda\vert W_{ba}^{G} \vert$ or 
$ \vert W_{ba}^{G}\vert=\lambda\vert W_{ab}^{G} \vert$, for some positive
rational number $ \lambda>1 $. Here, $ W_{ab}^{G}=W_{a\underline{1}b}^{G} $.
\\

The study of quasi-$\lambda$-DB graphs is only beginning (\cite{R1}, \cite{R2}).
Inspired by the notion of quasi-$\lambda$-DB graph together with the $n$-distance-balanced property introduced by 
Faghani, Pourhadi and Kharazi \cite{s} we present a new class of graphs as follows.
\begin{definition}\label{def1}
	A graph $ G $ is called \textit{quasi-($\lambda,n $)-distance-balanced} (for short, quasi-($\lambda,n $)-DB) if for each $a,b\in V(G)$ with $ d(a,b)=n $ we have either
	$ \vert W_{a\underline{n}b}^{G}\vert=\lambda\vert W_{b\underline{n}a}^{G} \vert$ or 
		$ \vert W_{b\underline{n}a}^{G}\vert=\lambda\vert W_{a\underline{n}b}^{G} \vert$, for some positive
		rational number $ \lambda>1 $. 
\end{definition}

For $ n=\lambda =1 $ the graph $ G $ is simply called \textit{distance-balanced}, which was initially introduced by Jerebic et al. \cite{a} and for $ n=1 $, $ G $ is called quasi-$\lambda $-distance-balanced graph defined by Abedi et al. \cite{R1}.
\\

The  paper is organized as follows. In the next section, 
 we investigate the quasi-$\lambda $-distance-balanced graphs and reveal some related facts, and then we  focus the problems recently arisen in Abedi et al. \cite{R1}. In Section 3, we initially introduce a new class of graphs which generalizes the quasi-$\lambda $-distance-balanced graphs and then present some result and a method to structure concerning with these graphs. Furthermore,  some problems and a conjecture for the further studies are included.



\section{Some facts of quasi-$\lambda $-distance-balanced graphs}
\label{S3}

In 2017, Abedi et al.  \cite{R1} introduced the notion of quasi-$\lambda $-distance-balanced graph which is the special case of 
quasi-($\lambda,n $)-distance-balanced graph by setting $ n=1 $. 
Since all examples of quasi-$\lambda $-DB graphs known to the authors are bipartite graphs, they arose the following natural question:

\begin{problem}[{\cite{R1}}]
	Does there exist a non-bipartite quasi-$\lambda $-DB graph?
\end{problem}
	In the following, we give the negative response for the above problem.
	
	\begin{theorem}
		If $ G $ is a connected quasi-$\lambda $-distance-balanced graph, then $ G $ is bipartite. 
	\end{theorem}
	 
	 \begin{proof}
	 	Inspired by the proof of \cite[Theorem 1.3]{R1}, let  $ G $ be  a quasi-$\lambda $-DB graph with $ d=\textrm{diam}(G) $, 
	 	and the vertex set $ \{v_{1}, v_{2}, \ldots, v_{2l+1} \} $ form an odd  circle with length $ 2l+1 $ such that $ v_{i}v_{i+1}\in E(G) $ and 	 	
$$A_{ij}=\bigg\{v\in V(G) \ | \ d(v,v_{i+k})=m_{jk}, \ m_{jk}=\{ 1,2, \ldots, d\},k=0,1,\ldots, 2l \bigg\}, \quad 2\leq j \leq r $$	
such that
$$ W_{v_{i}v_{i+1}}^{G}=\bigg(\bigcup_{j=1}^{r} A_{ij}\bigg)\cup \{ v_{i}, v_{i+2l}\} 
\quad W_{v_{i+1}v_{i}}^{G}=\bigg(\bigcup_{j=1}^{r} A_{(i+1)j}\bigg)\cup \{ v_{i+1}, v_{i+2}\} $$ 	
	 	where the calculations in indexes $ i $ are performed modulo $ 2l+1 $ and some $ r\in \mathbb{N} $. 
		 	Taking $ |A_{ij} |=a_{ij} $ for $ i=0,1,\ldots, 2l $ and $ j=1,2,\ldots, r $ and following the hypothesis there exist 
		 	$ e_{i}\in\{\pm1 \} $, $ i= 0,1,\ldots, 2l$, such that 
	\begin{eqnarray*}
	 &\sum_{j=1}^{r}a_{0j} +2 =\lambda^{e_{0}}\bigg(\sum_{j=1}^{r}a_{1j} +2\bigg),
	 \\& \sum_{j=1}^{r}a_{1j} +2 =\lambda^{e_{1}}\bigg(\sum_{j=1}^{r}a_{2j} +2\bigg),
	 \\& \vdots 
	 \\& \sum_{j=1}^{r}a_{(2l-1)j} +2 =\lambda^{e_{2l-1}}\bigg(\sum_{j=1}^{r}a_{(2l)j} +2\bigg),
	  \\& \sum_{j=1}^{r}a_{(2l)j} +2 =\lambda^{e_{2l}}\bigg(\sum_{j=1}^{r}a_{0j} +2\bigg).
	\end{eqnarray*}	 	
		 	Now, multiplying all $ (2l+1) $ equations above implies that $ \lambda^{\Sigma_{i=0}^{2l}e_{i}}=1 $, that is, 
		 	$ \Sigma_{i=0}^{2l} e_{i}=0 $. 
		 	On the other hand,  
		 	 $$ e_{i}\in\{\pm1 \}  \quad \Longrightarrow  \quad  1\leq |\Sigma_{i=0}^{2l} e_{i}|$$
		 	 which is a contradiction and hence $ G $ has no odd circle. This completes the proof.
	 \end{proof}
	 
	\begin{remark}\label{rem1}
	A bipartite graph has unbalanced parity if the two vertex sets are not the same size. 
	To describe more,	such a graph cannot be Hamiltonian, because a Hamilton circuit must alternate between the two vertex sets.
	\end{remark}

	\begin{remark}\label{rem2}
From Theorem 1.5 in \cite{R2} if $ G $ is a connected quasi-$\lambda $-distance-balanced graph with $ \delta(G)>1 $ then it is 2-connected 
and only stars are the quasi-$\lambda $-distance-balanced graphs with bridge. Moreover, stars are the only connected quasi-$\lambda $-distance-balanced graphs with $ \delta(G)=1 $. 
	\end{remark}
\subsection{Local operations} 
In this section we consider local operations on graphs and establish that they typically 
demolish the  quasi-$\lambda $-distance-balanced property.

\begin{theorem}\label{thm2.1}
If $ G $ is a connected quasi-$\lambda $-distance-balanced graph with $ \delta(G)>1 $, then for any adjacent edges $ e_{1}, e_{2}\in E(G) $ either $ G- e_{1} $ or $ G- e_{2} $ is not quasi-$\lambda $-distance-balanced graph. 
\end{theorem}

\begin{proof}
Let  $ e_{1}=ab, e_{2}=ac$ be adjacent edges in $ G  $,    
 without loss of generality and using Remark \ref{rem2} let $ c $ belong  to $ P_{1} $ as the shortest path connecting $ a $ to $ b $ in $ H_{1}=G-e_{1}. $
Suppose that $ x\in  W_{ca}^{G}$. Following the fact that $ e_{1} $ does not lie on any shortest
$(x,c)$-path in $ G $ we get $ d_{H_{1}}(x,c)=d_{G}(x,c) $. This implies that 
\begin{eqnarray*}
 d_{H_{1}}(x,a) \geq d_{G}(x,a)> d_{G}(x,c)= d_{H_{1}}(x,c)
\end{eqnarray*}
which shows that $ x\in W_{ca}^{H_{1}} $, that is, $ W_{ca}^{G}\subseteq W_{ca}^{H_{1}} $.
On the other hand, $ b\in W_{ac}^{G}\  \cup \ 
\underaccent{\mathlarger{a} \  \ \mathlarger{c}}{W}^{G}=W_{ac}^{G} $ and $ b\in  W_{ca}^{H_{1}}  $ 
which yields 
\begin{eqnarray}\label{eq2.1}
|W_{ca}^{H_{1}}| \geq |W_{ca}^{G} |+1.
\end{eqnarray}
Since $ G $ is a quasi-$\lambda $-distance-balanced graph we consider the following cases:
\begin{description}
	\item[Case 1.]  
  If $ |W_{ca}^{G} | =\lambda |W_{ac}^{G} | $ then using \eqref{eq2.1} we obtain
  $$ |W_{ca}^{H_{1}}| \geq |W_{ca}^{G} |+1 =\lambda |W_{ac}^{G} | +1 \geq \lambda |W_{ac}^{H_{1}} | +1 > \lambda |W_{ac}^{H_{1}} |.   $$

Hence, $ |W_{ca}^{H_{1}}| \neq \lambda |W_{ac}^{H_{1}} | $. Furthermore, if $ |W_{ac}^{H_{1}} | = \lambda |W_{ca}^{H_{1}}|  $ then we have
 $$ |W_{ca}^{H_{1}}| \geq |W_{ca}^{G} |+1 =\lambda |W_{ac}^{G} | +1 \geq \lambda |W_{ac}^{H_{1}} | +1 = 
 \lambda^{2} |W_{ca}^{H_{1}} | +1   $$
 which is a contradiction. Therefore, $ H_{1} $ is not quasi-$\lambda $-distance-balanced. 
 	\item[Case 2.]
 Now suppose that $ |W_{ac}^{G} | =\lambda |W_{ca}^{G} | $. Obviously,  $|W_{ac}^{H_{1}}|\neq \lambda |W_{ca}^{H_{1}}|$ since
  $$ \lambda|W_{ca}^{H_{1}}| \geq \lambda |W_{ca}^{G} |+\lambda = |W_{ac}^{G} | +1 \geq |W_{ac}^{H_{1}} | +1 > |W_{ac}^{H_{1}}|.$$
 For this case if $ H_{1} $ is not quasi-$\lambda $-distance-balanced, then the proof is complete, otherwise, let us consider $  |W_{ca}^{H_{1}}| =\lambda |W_{ac}^{H_{1}}| $. 
 \\
 
\item[Case 2.1.] 
If $  |W_{ba}^{G} | = \lambda |W_{ab}^{G} |$, then
 by the same reasoning for the edge $ e_{2} $ and considering a path $ P_{2} $ 
 for $ H_{2}=G-e_{2} $ we arrive at some equalities similar to the ones above. That is, 
 $$  |W_{ab}^{G} | =\lambda |W_{ba}^{G} |, \quad |W_{ba}^{H_{2}}| =\lambda |W_{ab}^{H_{2}}| $$
 which is a contradiction since it would imply that $ |W_{ab}^{G} | = |W_{ba}^{G} |=\frac{|V(G) |}{\lambda +1 } $. 
 \\
 
 \item[Case 2.2.] 
 Now suppose that  $  |W_{ab}^{G} | = \lambda |W_{ba}^{G} |$, then considering the edge $ e_{3}=bd\in E(G) $ for $ a\neq d $ 
 and a path $ P_{3} $ for $ H_{3}=G-e_{3} $  we similarly derive that 
 $$  |W_{ba}^{G} | =\lambda |W_{ab}^{G} |, \quad |W_{ab}^{H_{3}}| =\lambda |W_{ba}^{H_{3}}| $$
 which shows that $ G $ is DB-graph and this is a contradiction.
\end{description}
This completes the proof. 
\end{proof}
Let us denote the complete graph and the cycle of order $ n $ by $ K_n $ and $ C_n $, respectively. The complement or inverse of a graph $ G $ is a graph $ \overline{G} $ on the same vertices such that two vertices of $ \overline{G} $ are adjacent if and only if they are not adjacent in $ G $.
\begin{corollary}
	Let graph $ G$ with $ \delta(G)>1 $ be given. Suppose that  $ G+e_{i} $ are quasi-$\lambda $-DB graphs for $ i=1,2 $ and $ e_{1}, e_{2} $ are two adjacent edges of $ \overline{G} $.
	Then $ G + e_{1} + e_{2} $ is not quasi-$\lambda $-DB.
\end{corollary}

\begin{proof}
Set  $ H:= G + e_{1} + e_{2} $. Suppose that $ H $ is  a quasi-$\lambda $-distance-balanced graph. 
Then by Theorem \ref{thm2.1}, 
either $ H_{1}:=H-e_{1}$ or $ H_{2}:=H-e_{2}$, is not quasi-$\lambda $-distance-balanced, 
which contradicts the hypothesis. 
\end{proof}
		
The join $ G + H $ of graphs $ G $ and $ H $ with disjoint vertex sets $ V_1 $ and $ V_2 $ 
and edge sets $ E_1 $ and $ E_2 $ is the graph union $ G \cup H $ together with all the edges joining $ V_1 $ and $ V_2 $. 	
Since  quasi-$\lambda $-distance-balanced graphs are triangle free it is obvious that $ G \cup H $  has no such property for any 
nontrivial graphs $ G $ and $ H $. Hence,  $ G \cup H $ is quasi-$\lambda $-distance-balanced if and only if 
 $ G $ and $ H $ are empty graphs, that is, $ G \cup H =K_{m,n} $ where $ m,n $ are the order of $ G $ and $ H $, and $ m\neq n $.
\\
		
For a vertex $u$ of $G$ the \textit{total distance} $D_{G}(u)$ of $u$ is $D_{G}(u) =\Sigma_{v\in V(G)} d_{G}(u, v).$
Whenever $ G $ will be clear from the context we will write $ d(u, v) $ and $ D(u) $ instead of $ d_{G}(u, v)$ and $ D_{G}(u) $, respectively.

\begin{proposition}
Suppose that $ G $ is a quasi-$\lambda $-distance-balanced graph, then for any $ u,v\in V(G) $ with $ d(u,v)=2 $ we have 
$ D(u)+D(v) $ is even. 
\end{proposition}

\begin{proof}
	Suppose that $ u,v, w\in V(G) $ with $ d(u,v)=2 $ and  $ d(u,w)= d(w,v)=1.$
\end{proof}
 
 \begin{theorem}[{\cite{R5}}]
 Let $ G $ be a connected graph. Then $ G $ is distance-balanced if and only if $ |\{ D(u): u\in V(G) \} | =1,$ 
 that is, $ G $ is transmission-regular.
 \end{theorem}
Since any quasi-$\lambda $-distance-balanced graph is bipartite, it seems we have only two values for $ | W_{uv}^{G}| $.
Using this together with the structure of known such graphs that we suspect the following is true. 
\begin{conjecture}\label{conj1}
If  $ G $ is quasi-$\lambda $-distance-balanced graph then $$ |\{ deg(u): u\in V(G) \} |= |\{ D(u): u\in V(G) \} | =2.$$
\end{conjecture}

\subsection{Quasi-$\lambda $-DB and some graph products}

The \textit{corona product} $ G \circ H $ is obtained by taking one copy of $ G $ and $ |V(G)| $ copies of $ H $; and by joining each vertex of the $ i $-th copy of $ H $ to the $ i $th vertex of $ G $, $ i = 1, 2, \ldots, |V(G)|. $ 

\begin{theorem}
The corona product of two arbitrary, nontrivial graphs $ G $ and $ H $ 
is 	quasi-$\lambda $-DB if and only if $ G $  and $ H $ are empty graphs. 
Moreover, $ \lambda = |V(H)| $.
\end{theorem}

\begin{proof}
Suppose that $ G $ and $ H $ are arbitrary, nontrivial and connected graphs, and let $ H_i $ be the $ i $-th copy of
$ H $, where $ i=1,2,\ldots, |V(G)|  $. Assume that $ G \circ H $ is quasi-$\lambda $-DB and $ uv \in E(G\circ H) $ such that $ u\in V(G) $ and $ v\in V(H_i) $. Hence, we get
\begin{eqnarray*}
	|W_{uv}^{G\circ H} | = \lambda |W_{vu}^{G\circ H} |=\lambda = |V(G)|(|V(H)|+1)-\textrm{deg}_{G\circ H}(v).
\end{eqnarray*}
 which implies that $ H $ must be regular. On the other hand if $ u v\in E(H_i) $ then
 \begin{eqnarray*}
 	|W_{uv}^{G\circ H} | = 	|W_{uv}^{H} |<\lambda, \quad   |W_{vu}^{G\circ H} | = |W_{vu}^{ H} |<\lambda
 \end{eqnarray*}
which is a contradiction unless $ H $ is an empty graph. Now suppose that $ u  v \in E(G) $ then 
\begin{eqnarray*}
	|W_{uv}^{G\circ H} | = 	|W_{uv}^{G} |+|V(H)| <\lambda, \quad   |W_{vu}^{G\circ H} | = |W_{vu}^{G} | +|V(H)|<\lambda
\end{eqnarray*}
which contradicts that $ G \circ H $ is a quasi-$\lambda $-DB graph, hence $ G $ is an empty graph. 
That is,  $ G \circ H $ is disconnected and formed by $ |V(G)| $ disjoint stars $ S_{k} $ where $ k=|V(H)|. $ 
The converse is obvious and so the consequence follows. 
\end{proof}

Very recently, a problem concerning with characterizing the quasi-$\lambda $-DB direct products 
has been arisen by Abedi et al. \cite{R1}.
Throughout this section, we present some facts regarding with quasi-$\lambda $-DB direct products 
which can be helpful in further investigations. Recall that two vertices $(u_{1}, u_{2}), (v_{1}, v_{2})\in  V(G)\times V(H)$  
are adjacent in $ G\times H $ when $ u_{1}v_{1}\in E(G) $ and $ u_{2}v_{2}\in E(H) $.

\begin{remark}
The graphs $G,H $ have a triangle, and more general, have an odd cycle with same type if and only if the direct product $ G\times H $ has. Therefore, for this case $ G\times H $ cannot be quasi-$\lambda $-distance-balanced.
\end{remark}

\begin{remark}
	If $ G\times H $ is connected then $ G $ or $ H $ is not quasi-$\lambda $-distance-balanced. Therefore, 
	there is no connected quasi-$\lambda $-distance-balanced graph $ G\times H $ while both $G $ and $H $ are quasi-$\lambda $-distance-balanced. In other words, if  $G $ and $H $ are quasi-$\lambda $-distance-balanced, then $ G\times H $ is disconnected. 
	Moreover, a sufficient condition for  the connectedness of $ G\times H $ is that $ |E(G)|+|E(H)| $ must be even (see also \cite{R6}). 
\end{remark}

In the following we define a new concept related to regularity in graphs to present a result concerning with 
quasi-$\lambda $-distance-balanced property of $ G\times H $. First, suppose that
 $ \textrm{Deg}(G) $ denotes the set of all distinct degrees observed in $ G. $

\begin{definition}\label{def2}
A graph $ G $ is said to be ($k_{1},k_{2}$)-\textit{regular} if $ \textrm{Deg}(G)=\{k_{1}, k_{2}\} $ and no adjacent vertices have the same degree.
\end{definition}
\begin{remark}\label{rem3}
According to Conjecture \ref{conj1}, we conjecture that any quasi-$\lambda $-distance-balanced graph is ($k_{1},k_{2}$)-regular 
for some $ k_{1}, k_{2}\in \mathbb{N}. $
\end{remark}
Denoted by $ D_{i,j}^{G}(x,y) $ we mean
\begin{eqnarray*}
 D_{i,j}^{G}(x,y)=\{u\in V(G) \ | \ d(u,x)=i, d(u,y)=j   \}.
\end{eqnarray*}
The consequence of the following result maybe useful for the future studies. 
\begin{proposition}
	Suppose $ G, H $ are  $(r_{1},\overline{r}_{1})$-regular and $(r_{2},\overline{r}_{2})$-regular graphs, respectively, where 
	 $ r_{1}>\overline{r}_{1} $ and $ \overline{r}_{2}> r_{2} $ with 
	 $ r_{1}+r_{2}=\overline{r}_{1}+ \overline{r}_{2} $ and $ \textrm{diam}(G)=\textrm{diam}(H)=3. $
	Also, let $ G, H $ be quasi-$\lambda $-distance-balanced such that  
	$$\lambda=\dfrac{n_{G}+n_{H}}{2(r_{1}+r_{2})}>1, \qquad  n_{G}:=|V(G)|>n_{H}:=|V(H)|.$$
	Then there is no $\lambda^{*}>1$ in which the graph $ G\times H $ is quasi-$\lambda^{*} $-distance-balanced.
\end{proposition}

\begin{proof}
Following the Definition \ref{def2} and hypotheses, assume that 
$$ \text{deg}_{G}(x)=r_{1}, \ \text{deg}_{G}(y)=\overline{r}_{1}, \ \text{deg}_{H}(a)=r_{2}, \ \text{deg}_{H}(b)=\overline{r}_{2}.  $$
Then, without loss of generality, we get
\begin{eqnarray}\label{eq2.2}
\begin{split}
&|W_{xy}^{G}| = r_{1} + |D_{2,3}^{G}(x,y) |, \quad |W_{yx}^{G}|=\overline{r}_{1} +|D_{2,3}^{G}(y,x) |, 
\\&|W_{ab}^{H}| = r_{2} + |D_{2,3}^{G}(a,b) |, \quad |W_{ba}^{H}|=\overline{r}_{2} +|D_{2,3}^{G}(b,a) |, 
\end{split}
\end{eqnarray}
for any edges $ xy \in E(G) $ and $ ab \in E(H). $ Now, consider the case 
\begin{eqnarray}\label{eq2.7}
|W_{xy}^{G}| = \lambda |W_{yx}^{G}|=\dfrac{n_{G}}{\lambda +1}, \quad  |W_{ab}^{H}| = \lambda |W_{ba}^{H}|=\dfrac{n_{H}}{\lambda +1} 
\end{eqnarray}
regarding the assumption that $ G, H $ are quasi-$\lambda $-distance-balanced and so bipartite.
\\

For any arbitrary fixed edges $ xy \in E(G) $ and $ ab \in E(H) $, let $ (u,v)\in W_{(x,a)(y,b)}^{G\times H} $ then 
\begin{eqnarray*}
 d_{G\times H}((u,v),(x,a))< d_{G\times H}((u,v),(y,b)). 
\end{eqnarray*}
We know that for every pair of vertices $ (r,s), (t,w)\in V(G\times H) $ 
we have that $ d_{G\times H}((r,s), (t,w))= i $ if and only if either $d_{G}(r,t)=i$ and $i-d_{H}(s,w)$ is a nonnegative even number or  $d_{H}(s,w)=i$ and $i-d_{G}(r,t)$ is a nonnegative even number 
(see also \cite[Lemma 1.1]{R5}). Now, since   $ \textrm{diam}(G)=\textrm{diam}(H)=3 $ we have
\begin{eqnarray*}
	d_{G\times H}((u,v),(x,a))\in \{1,2 \}, \qquad d_{G\times H}((u,v),(y,b))\in \{ 2,3\}. 
\end{eqnarray*}
If $ d_{G\times H}((u,v),(y,b))=3 $ then
\begin{eqnarray}\label{eq2.3}
  d_{G}(u,y)=3, \ d_{H}(v,b)\in \{1,3 \}   \qquad \text{or} \qquad   d_{H}(v,b)=3, \  d_{G}(u,y)\in \{1,3 \}.
\end{eqnarray}
On the other hand, $d_{G\times H}((u,v),(x,a))\in \{1,2 \} $ implies that 
\begin{eqnarray}\label{eq2.4}
	ux \in E(G), \ va \in E(H)
 \quad \text{or} \quad 	u=x, \ d_{H}(v,a)=2   \quad \text{or} \quad  
 v=a, \  d_{G}(u,x)=2.
\end{eqnarray}
Eqs. \eqref{eq2.3} and \eqref{eq2.4} imply the following cases:
 \begin{eqnarray*}
  	\bigg(u=x, \ d_{H}(v,a)=2\bigg), \ \bigg(d_{H}(v,b)=3, \  d_{G}(u,y)\in \{1,3 \}\bigg)  \ \Longrightarrow \  
  	(u,v)\in \{ x\}\times D_{2,3}^{H}(a,b)
 \end{eqnarray*}
or 
\begin{eqnarray*}
	\bigg(v=a, \  d_{G}(u,x)=2 \bigg), \ \bigg(d_{G}(u,y)=3, \ d_{H}(v,b)\in \{1,3 \} \bigg)  \ \Longrightarrow \  
	(u,v)\in D_{2,3}^{G}(x,y) \times \{ a\}.
\end{eqnarray*}
That is, 
\begin{eqnarray}\label{eq2.5}
 W_{(x,a)(y,b)}^{G\times H} \subset 
  \bigg(\{ x\}\times D_{2,3}^{H}(a,b)\bigg) \bigcup \bigg(D_{2,3}^{G}(x,y) \times \{ a\}\bigg).
\end{eqnarray} 
One can easily show that the converse of \eqref{eq2.5} is also true, thus 
\begin{eqnarray}\label{eq2.6}
\begin{split}
 &W_{(x,a)(y,b)}^{G\times H} =
\bigg(\{ x\}\times D_{2,3}^{H}(a,b)\bigg) \bigcup \bigg(D_{2,3}^{G}(x,y) \times \{ a\}\bigg),
\\& W_{(y,b)(x,a)}^{G\times H} =
\bigg(\{ y\}\times D_{3,2}^{H}(a,b)\bigg) \bigcup \bigg(D_{3,2}^{G}(x,y) \times \{ b\}\bigg).
\end{split}
\end{eqnarray}
The equalities \eqref{eq2.6} implies that 
\begin{eqnarray}\label{eq2.12}
 \bigg | W_{(x,a)(y,b)}^{G\times H}\bigg | =\bigg | D_{2,3}^{H}(a,b) \bigg | +\bigg |  D_{2,3}^{G}(x,y) \bigg |, \quad 
 \bigg |W_{(y,b)(x,a)}^{G\times H} \bigg |= \bigg |D_{3,2}^{H}(a,b) \bigg |+ \bigg |D_{3,2}^{G}(x,y)\bigg |.
\end{eqnarray}
This together with \eqref{eq2.2} and \eqref{eq2.7} imply that 
\begin{eqnarray*}
	\bigg | W_{(x,a)(y,b)}^{G\times H}\bigg | = \bigg(\dfrac{n_{G}}{\lambda +1} - r_{1}\bigg) + \bigg(\dfrac{n_{H}}{\lambda +1}-r_{2}\bigg), \quad \bigg |W_{(y,b)(x,a)}^{G\times H} \bigg |=
	\bigg(\dfrac{n_{G}}{\lambda(\lambda +1)} - \overline{r}_{1}\bigg) + \bigg(\dfrac{n_{H}}{\lambda(\lambda +1)}-\overline{r}_{2}\bigg)
\end{eqnarray*}
which shows that 
\begin{eqnarray}\label{eq2.8}
	\bigg | W_{(x,a)(y,b)}^{G\times H}\bigg | =\lambda_{1} \bigg |W_{(y,b)(x,a)}^{G\times H} \bigg | 
	\quad \text{where} \quad 
	\lambda_{1}=\dfrac{\lambda(n_{G}+n_{H})-(r_{1}+r_{2})\lambda(\lambda+1)}
	{(n_{G}+n_{H})-(\overline{r}_{1}+\overline{r}_{2})\lambda(\lambda+1)}>1.
\end{eqnarray}
Now, consider the second case 
\begin{eqnarray}\label{eq2.9}
|W_{yx}^{G}| = \lambda |W_{xy}^{G}|=\dfrac{n_{G}}{\lambda +1}, \quad  |W_{ab}^{H}| = \lambda |W_{ba}^{H}|=\dfrac{n_{H}}{\lambda +1}. 
\end{eqnarray}
Considering this together with \eqref{eq2.2} and \eqref{eq2.12} we obtain
\begin{eqnarray*}
	\bigg | W_{(x,a)(y,b)}^{G\times H}\bigg | = \bigg(\dfrac{n_{G}}{\lambda(\lambda +1)} - r_{1}\bigg) + \bigg(\dfrac{n_{H}}{\lambda +1}-r_{2}\bigg), \quad \bigg |W_{(y,b)(x,a)}^{G\times H} \bigg |=
	\bigg(\dfrac{n_{G}}{\lambda +1 } - \overline{r}_{1}\bigg) + \bigg(\dfrac{n_{H}}{\lambda(\lambda +1)}-\overline{r}_{2}\bigg)
\end{eqnarray*}
which means that 
\begin{eqnarray}\label{eq2.10}
  \bigg |W_{(y,b)(x,a)}^{G\times H} \bigg | =\lambda_{2} \bigg | W_{(x,a)(y,b)}^{G\times H}\bigg | 
\quad \text{where} \quad 
\lambda_{2}=\dfrac{(\lambda n_{G}+n_{H})-(\overline{r}_{1}+\overline{r}_{2})\lambda(\lambda+1)}
{(n_{G}+\lambda n_{H})-(r_{1}+r_{2})\lambda(\lambda+1)}>1.
\end{eqnarray}
For the third case, suppose that 
\begin{eqnarray}\label{eq2.11}
|W_{xy}^{G}| = \lambda |W_{yx}^{G}|=\dfrac{n_{G}}{\lambda +1}, \quad  |W_{ba}^{H}| = \lambda |W_{ab}^{H}|=\dfrac{n_{H}}{\lambda +1}. 
\end{eqnarray}
Then using \eqref{eq2.2} and \eqref{eq2.12} we get 
\begin{eqnarray*}
	\bigg | W_{(x,a)(y,b)}^{G\times H}\bigg | = \bigg(\dfrac{n_{G}}{\lambda +1} - r_{1}\bigg) + \bigg(\dfrac{n_{H}}{\lambda(\lambda +1) }-r_{2}\bigg), \quad \bigg |W_{(y,b)(x,a)}^{G\times H} \bigg |=
	\bigg(\dfrac{n_{G}}{\lambda(\lambda +1)} - \overline{r}_{1}\bigg) + \bigg(\dfrac{n_{H}}{\lambda +1 }-\overline{r}_{2}\bigg).
\end{eqnarray*}
Therefore, 
\begin{eqnarray}\label{eq2.13}
\bigg | W_{(x,a)(y,b)}^{G\times H}\bigg | =\lambda_{2} \bigg |W_{(y,b)(x,a)}^{G\times H} \bigg |.
\end{eqnarray}
For the last case let us take 
\begin{eqnarray}\label{eq2.14}
|W_{yx}^{G}| = \lambda |W_{xy}^{G}|=\dfrac{n_{G}}{\lambda +1}, \quad  |W_{ba}^{H}| = \lambda |W_{ab}^{H}|=\dfrac{n_{H}}{\lambda +1}. 
\end{eqnarray}
Then using \eqref{eq2.2} and \eqref{eq2.13} we get 
\begin{eqnarray*}
	\bigg | W_{(x,a)(y,b)}^{G\times H}\bigg | = \bigg(\dfrac{n_{G}}{\lambda(\lambda +1)} - r_{1}\bigg) + \bigg(\dfrac{n_{H}}{\lambda(\lambda +1) }-r_{2}\bigg), \quad \bigg |W_{(y,b)(x,a)}^{G\times H} \bigg |=
	\bigg(\dfrac{n_{G}}{\lambda +1} - \overline{r}_{1}\bigg) + \bigg(\dfrac{n_{H}}{\lambda +1 }-\overline{r}_{2}\bigg).
\end{eqnarray*}
which implies that 
\begin{eqnarray}\label{eq2.15}
   \bigg |W_{(y,b)(x,a)}^{G\times H} \bigg | = \lambda_{1} \bigg | W_{(x,a)(y,b)}^{G\times H}\bigg |.
\end{eqnarray}
Moving forward, one can see that $  \lambda_{1}= \lambda_{2} $, since 
\begin{eqnarray*} 
\begin{split}
 \lambda_{1}= \lambda_{2} &\quad \Longleftrightarrow   \quad 
 \dfrac{\lambda(n_{G}+n_{H})-(r_{1}+r_{2})\lambda(\lambda+1)}
 {(n_{G}+n_{H})-(\overline{r}_{1}+\overline{r}_{2})\lambda(\lambda+1)}
 =\dfrac{(\lambda n_{G}+n_{H})-(\overline{r}_{1}+\overline{r}_{2})\lambda(\lambda+1)}
 {(n_{G}+\lambda n_{H})-(r_{1}+r_{2})\lambda(\lambda+1)}
\\ & \quad \Longleftrightarrow   \quad \lambda=\dfrac{n_{G}+n_{H}}{2(r_{1}+r_{2})}.
\end{split}
\end{eqnarray*}
Hence, $ G\times H $ can not be quasi-$\lambda^{*} $-distance-balanced graph if 
$$\lambda^{*} = \lambda_{1}= \lambda_{2}=-1 $$
which contradicts, and the proof is complete.
\end{proof}
The recent result is valuable because of the following remark. 
\begin{remark}
	Following Conjecture \ref{conj1} and Remark \ref{rem3}, it seems
	$ G\times H $ cannot be a quasi-$\lambda^{*} $-distance-balanced graph for any $\lambda^{*}>1$ whenever 
	 $ G, H $ are quasi-$\lambda $-distance-balanced graphs with diameter 3. 
\end{remark}


\subsection{A method to construct quasi-$\lambda $-distance-balanced graphs}

In the following we first improve the method presented by Abedi et al. \cite{R1} to obtain the quasi-$\lambda $-distance-balanced graphs
and then using a new technique we generate the quasi-$\lambda $-distance-balanced graphs with 
arbitrary diameter. 
\\

Given the simple graphs $ G_{1}, G_{2} $, by the symbol $ G_{1}*G_{2} $ we mean 
$$ V(G_{1}*G_{2})= V(G_{1})\cup V(G_{2}),  \qquad  
E(G_{1}*G_{2})=\{ab \ |  \ a\in V(G_{1}), b\in V(G_{2}) \} \cup  E(G_{1})\cup E(G_{2}).  $$

Moreover, notation $\overline{G_{1}*G_{2}* \cdots *G_{k}}$ stands for the graph 
$ (G_{1}*G_{2}* \cdots *G_{k})\cup (G_{1}*G_{k})  $  which looks like a cycle (see Figure 4).  

\begin{definition}
	Let $ G $ be a non-empty  $ (t_{1}, t_{2})$-biregular bipartite graph of order $ (n_{1}+n_{2}) $ with bipartition sets $ B $ and $ C $ with sizes $ n_{1} $ and $ n_{2} $, and degree sets $ \{ t_{1}\} $ and $ \{t_{2} \} $, respectively. Let $ m $ and $ k $ be non-zero integers with $ 1  \leq n_{1} + m \leq n_{2}+k $. 
	Let $ A $ and $ D $ be sets of size $ m $ and $ k $, respectively. 
	Graph $ H(m, G, k) $ is defined as the
	graph with vertex set $ V(H) = A \cup B \cup C \cup D, $ and edge set $ E(H)=E(A*B)\cup E(G)\cup E(C*D)  $ (see  Figure 1).  
\end{definition}
For the case $ t=t_{1}= t_{2} $ and $ n=n_{1}=n_{2} $, graph $ Ala(m, G, k) $ in \cite[Definition 3.1]{R1} is obtained.
	\begin{figure}[H]\label{pic5}\footnotesize 
		\begin{center}
			\hskip 0.5 cm\includegraphics[scale=0.6]{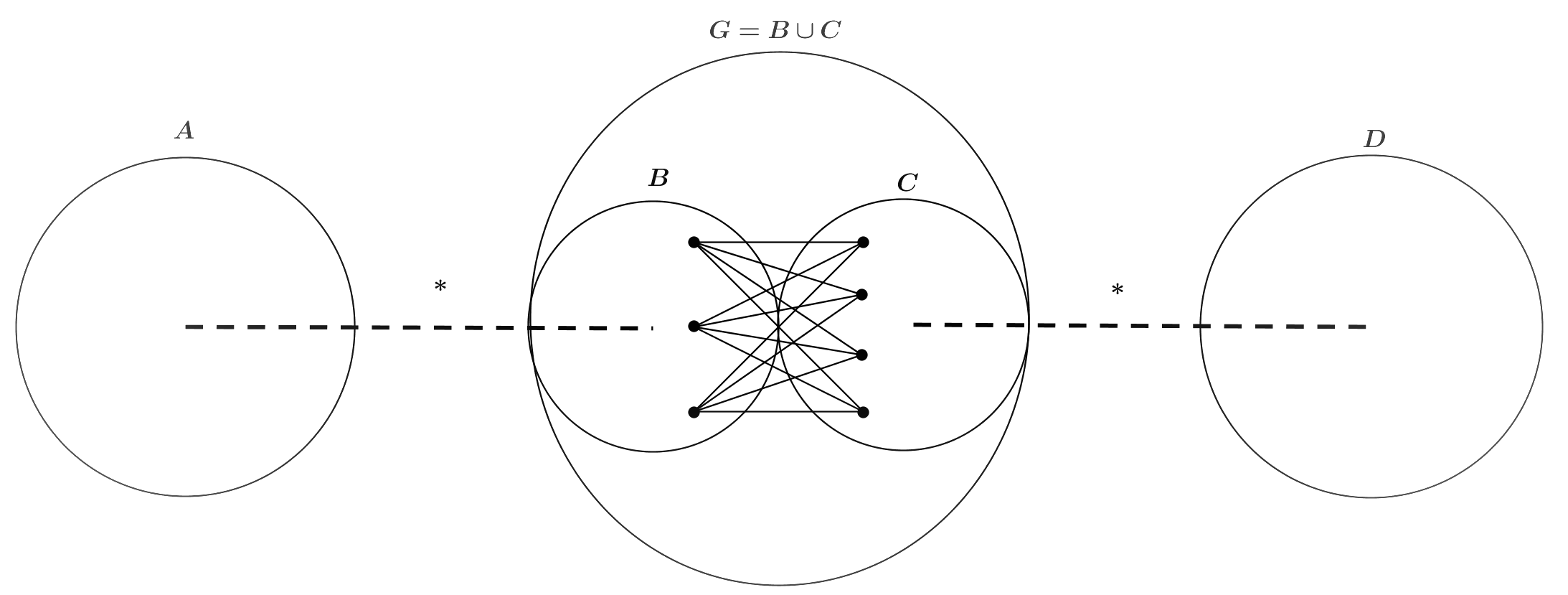}
		\end{center}
		\caption{\footnotesize $ (3, 4)$-regular bipartite graph $ G=B\cup C $ and the sets $ A, D $ in the structure of $ H(m, G, k) $.  \normalsize}
	\end{figure}\normalsize

\begin{proposition}\label{pro1}
Let $ G $ be a non-empty $ (t_{1}, t_{2})$-regular bipartite graph of order $ (n_{1}+n_{2}) $. The graph $ H(m, G, k) $ is 
quasi-$\lambda $-DB if and only if 
$ n_{2}+k \neq n_{1} + m$, $ t_{1}=n_{2}-m$ and $ t_{2}=n_{1}-k $, and $ \lambda=\frac{n_{2}+k}{n_{1} + m} $. Moreover, if  $ t_{1}= t_{2}=n_{1}-k=n_{2}-m $, 
then $ H(m, G, k) $ is DB.
\end{proposition}

\begin{proof}
	In view of the proof of \cite[Proposition 3.2]{R1}, for any adjacent vertices $ u $ and $ v $ in $ V(G) $, 
	$ W_{uv}^{H}\cup W_{vu}^{H} =V(H) $ since $ H $ is bipartite, suppose that $ a\in A $  and $ b \in B $  are taken arbitrarily, then it is easily seen that 
	\begin{eqnarray*} 
		\begin{split}
&|W_{ab}^{H} | = 1+(n_{1}-1)	+(n_{2}-t_{1})+0 = n_{1}+n_{2}-t_{1},
 \\&|W_{ba}^{H} | = 1+(m-1)+t_{1}+k =m+k+t_{1}.
		\end{split}
	\end{eqnarray*}  
Moving forward, letting $ b\in B $, $ c\in C $ and $ d\in D $ with $ bc\in E(G), $ we get 
  \begin{eqnarray*} 
  	\begin{split}
  	&	|W_{bc}^{H} | =m+ (1+n_{1}-t_{2})+(t_{1}-1)+0 = m+ n_{1}+t_{1}-t_{2}, 
  	\\&	|W_{cb}^{H} | =k+ (1+n_{2}-t_{1})+(t_{2}-1)+0 = k+ n_{2}+t_{2}-t_{1},
  	\\&	|W_{cd}^{H} | =1+(k-1)+n_{2}+m = k+m+ t_{2},
  	\\&	|W_{dc}^{H} | =1+  (n_{2}-1)+(n_{1}-t_{2})+0=n_{1}+n_{2}-t_{2}.
  	\end{split}
  \end{eqnarray*} 
Hence, if $ H(m, G, k) $ is a quasi-$\lambda $-DB graph then 
 \begin{eqnarray}\label{eq2.16} 
	\begin{split}
	\dfrac{n_{1}+n_{2}-t_{1}}{ k+m+ t_{1}}, \ \dfrac{n_{1}+n_{2}-t_{2}}{ k+m+ t_{2}}, \ \dfrac{m+ n_{1}+t_{1}-t_{2}}{k+ n_{2}+t_{2}-t_{1}} 
	\in \{\lambda, \lambda^{-1}\},
	\end{split}
\end{eqnarray} 
which occurs only if 
$ t_{1}=n_{2}-m $ and $ t_{2} =n_{1}-k $, and then following \eqref{eq2.16} one can see that  $ \lambda $ would take the form 
$ \lambda=\frac{n_{2}+k}{n_{1} + m} $. For the converse, considering the discussion above, the desired conclusion follows. 
Furthermore, $ H(m, G, k) $ is DB if and only if each fraction in \eqref{eq2.16} is identically equal to 1, that is, 
it must be  $ t_{1}= t_{2}=n_{1}-k=n_{2}-m $. Therefore, the proof is complete.
\end{proof}
In the following, Figure 2 shows a quasi-$\frac{8}{7} $-DB graph $ H(3, G, 2)=3K_{1}*G*2K_{1} $ 
where $ G $ is a  $ (3, 2)$-biregular bipartite graph. 
And this is because of the fact that $ t_{1}=n_{2}-m=3 $ and $ t_{2} =n_{1}-k=2 $.
Moreover, 
$ |W_{ab}^{H} | \in \{7,8\}$ for any edge $ab\in E(H).$

\begin{remark}
In view of the assumptions of Proposition \ref{pro1}, 
notice that for the case $ t_{1}=t_{2} $ one can simply derive that $ n_{1}=n_{2} $ and then $ k=m. $	
\end{remark}

	\begin{figure}[H]\label{pic6}\footnotesize 
	\begin{center}
		\hskip 0.5 cm\includegraphics[scale=0.35]{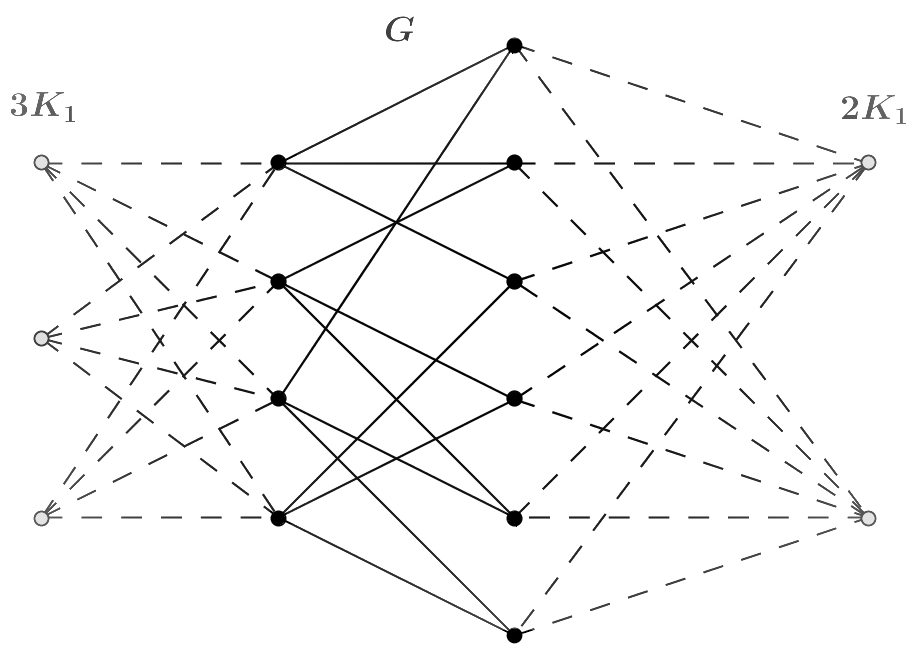}
	\end{center}
	\caption{\footnotesize $ H(3, G, 2) $ graph with $ (3, 2)$-regular bipartite graph $ G$. \normalsize}
\end{figure}\normalsize

In Figures 3-4, using the empty graphs $ mK_{1} $ and $ nK_{1} $, and the operation $ * $ we present a 
new class of quasi-$\lambda $-distance-balanced graphs with diameter 2. Here, for any $ x\in V(nK_{1}) $ 
and $ y\in V(mK_{1}) $ we see that 
\begin{eqnarray*} 
	\begin{split}
	 &|W_{xy}^{G_{1}}| = \lambda |W_{yx}^{G_{1}}|,  \quad \lambda =\dfrac{2m}{n} \quad  (\text{if} \ 2m>n), 
\quad \text{and} \quad 	 |W_{yx}^{G_{1}}| = \lambda  |W_{xy}^{G_{1}}|,  \quad \lambda =\dfrac{n}{2m} \quad  (\text{if} \ 2m<n),
  \\& |W_{xy}^{G_{2}}| = \lambda |W_{yx}^{G_{2}}|,  \quad \lambda =\dfrac{m}{n} \quad  (\text{if} \ m>n),
  \quad \text{and} \quad 	 |W_{yx}^{G_{2}}| = \lambda  |W_{xy}^{G_{2}}|,  \quad \lambda =\dfrac{n}{m} \quad  (\text{if} \ m<n).
	\end{split}
\end{eqnarray*}

\begin{figure}[H]\label{pic1}\footnotesize 
	\begin{center}
		\hskip 0.5 cm\includegraphics[scale=0.35]{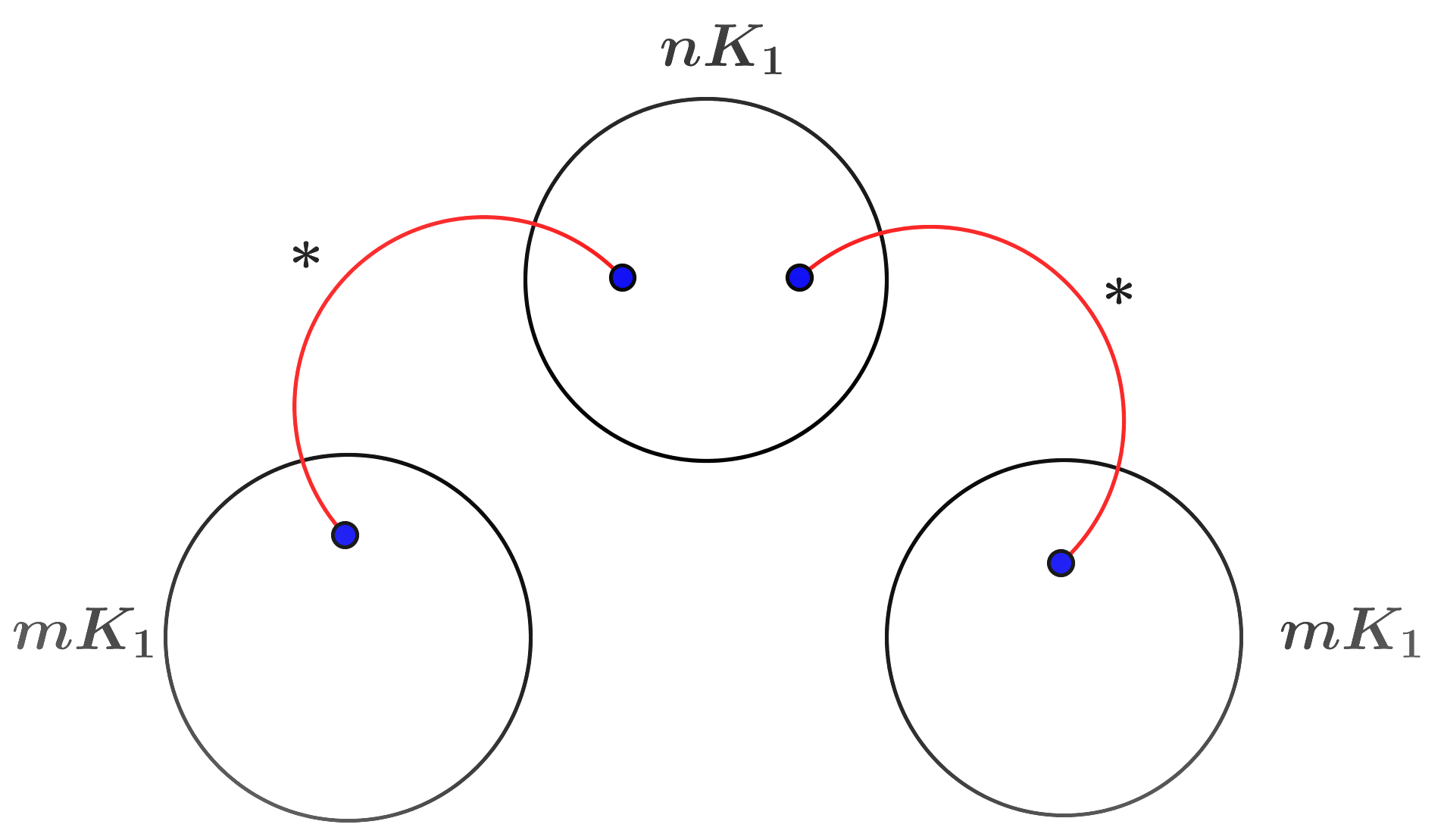}
	\end{center}
	\caption{\footnotesize Graph $ G_{1}= mK_{1} * nK_{1} * mK_{1} $, $ n\neq 2m $, is a quasi-$\lambda $-DB graph with $ d=2 $.  \normalsize}
\end{figure}\normalsize

\begin{figure}[H]\label{pic2}\footnotesize 
	\begin{center}
		\hskip 0.5 cm\includegraphics[scale=0.35]{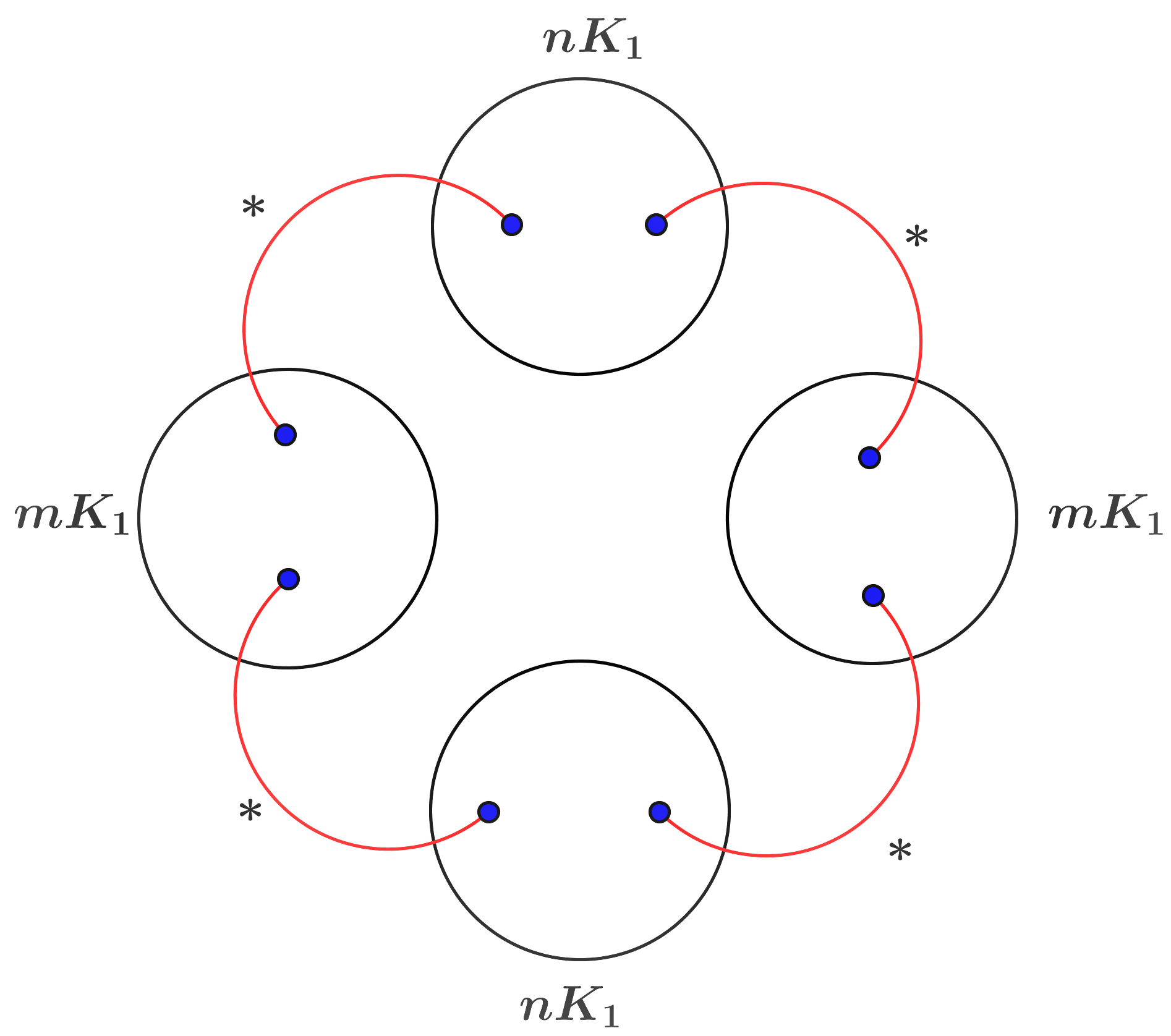}
	\end{center}
	\caption{\footnotesize Graph $ G_{2}= \overline{mK_{1} * nK_{1} * mK_{1}* nK_{1}} $, $ n\neq m $, is a quasi-$\lambda $-DB graph with $ d=2 $.  \normalsize}
\end{figure}\normalsize
In the following figure, graph $ G_{3}= \underbrace{\overline{mK_{1} * nK_{1} * mK_{1}*\cdots * mK_{1}* nK_{1}}}_{2d-\textrm{times}}  $, $ n\neq m $, is a  quasi-$\lambda $-DB graph with diameter $ d$ and $\lambda =\dfrac{m}{n}$ if $m>n$. 
\begin{figure}[H]\label{pic3}\footnotesize 
	\begin{center}
		\hskip 0.5 cm\includegraphics[scale=0.35]{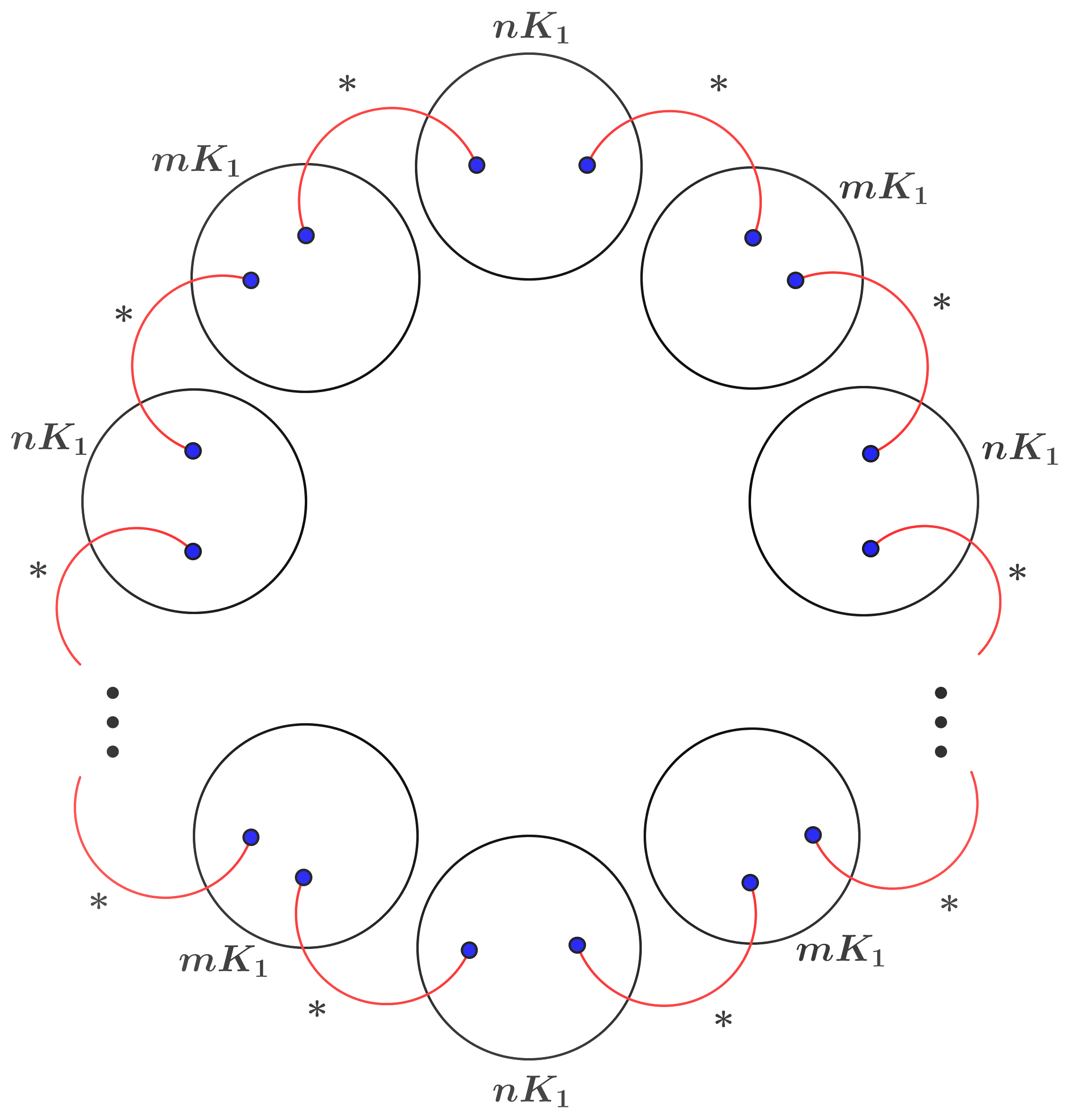}
	\end{center}
	\caption{\footnotesize Graph $ G_{3}$ with diameter $ d$.  \normalsize}
\end{figure}\normalsize

In Figure 6, we exemplify the method illustrated above for 
$$ G_{4}= \overline{K_{1} * 2K_{1} * K_{1}* 2K_{1} * K_{1}* 2K_{1}* K_{1}* 2K_{1}},$$
$$ G_{5}= \overline{2K_{1} * 3K_{1} * 2K_{1}* 3K_{1} * 2K_{1}* 3K_{1}}$$
as  quasi-$\frac{7}{5} $-DB and  quasi-$ \frac{8}{7} $-DB graphs, respectively.

\begin{figure}[H]\label{pic4}\footnotesize 
	\begin{center}
		\hskip 0.5 cm\includegraphics[scale=0.25]{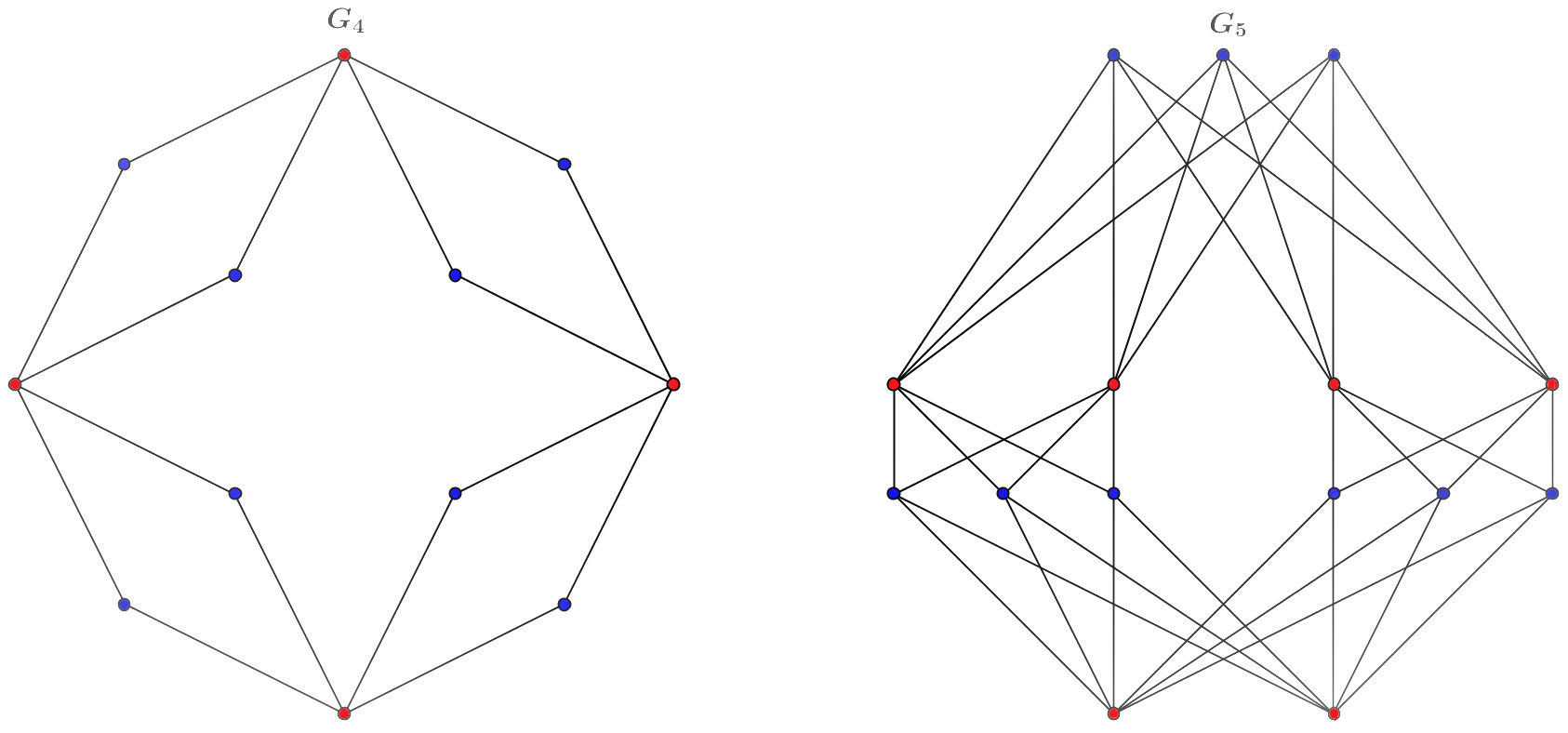}
	\end{center}
	\caption{\footnotesize Graphs $ G_{4} $ with diameter $ d=4$ and $ G_{5} $ with diameter $ d=3$.  \normalsize}
\end{figure}\normalsize
\begin{remark}
Considering the two methods above we see that the structured graphs are biregular bipartite. The quasi-$\frac{n_{2}+k}{n_{1} + m} $-DB graph $ H(m, G, k) $ in Proposition \ref{pro1} is an ($n_{1},n_{2} $)-biregular bipartite. 
Further, the quasi-$\dfrac{m}{n} $-DB graph $ \underbrace{\overline{mK_{1} * nK_{1} * mK_{1}*\cdots * mK_{1}* nK_{1}}}_{2d\textrm{-times}}$, with diameter $ d$, $ m>n $, is a ($2m,2n $)-biregular bipartite.
\end{remark}

\begin{problem}\label{pr1}
For $ m \neq n $, is there any ($m, n $)-biregular bipartite graph with no quasi-$ \lambda$-DB property?
\end{problem}
We know that any connected edge-transitive graph which is not DB, is a quasi-$ \lambda$-DB graph (see \cite{R1}) but how about the converse:
\begin{problem}\label{pr2}
Is there any  quasi-$ \lambda$-DB graph without edge-transitivity?
\end{problem}
If the answer of Problem \ref{pr2} is negative then by the fact that every edge-transitive graph (disallowing graphs with isolated vertices) that is not also vertex-transitive must be biregular, we find that quasi-$ \lambda$-DB graphs are exactly the 
biregular bipartite graphs which also solves the Problem \ref{pr1}. We also note that quasi-$ \lambda$-DB graphs are not  
vertex-transitive, since any graph with vertex-transitivity must be distance-balanced.

\section{Quasi-($\lambda,n $)-distance-balanced graphs}
Throughout this section, we define a new class of graphs which is considered as a generalization of quasi-$ \lambda$-DB graphs.
For $ a,b \in V(G) $ with $ d(a,b)=n $, let $ W_{a\underline{n}b}^{G} $ and 
$ \underaccent{\mathlarger{a} \ \underline{\mathlarger{n}} \ \mathlarger{b}}{W}^{G} $ 
be the sets of vertices as given in the first section.
Then $ G $ is called \textit{quasi-$(\lambda,n) $-distance-balanced} graph, 
in which either
$ \vert W_{a\underline{n}b}^{G}\vert=\lambda\vert W_{b\underline{n}a}^{G} \vert$ or 
$ \vert W_{b\underline{n}a}^{G}\vert=\lambda\vert W_{a\underline{n}b}^{G} \vert$, for some positive
rational number $ \lambda>1 $. 
\\

In the following using the graph operation $ * $ and the complete graphs we present a class of quasi-$(\lambda,n) $-distance-balanced graphs for any $ n $.

\begin{figure}[H]\label{pic7}\footnotesize 
	\begin{center}
		\hskip 0.5 cm\includegraphics[scale=0.35]{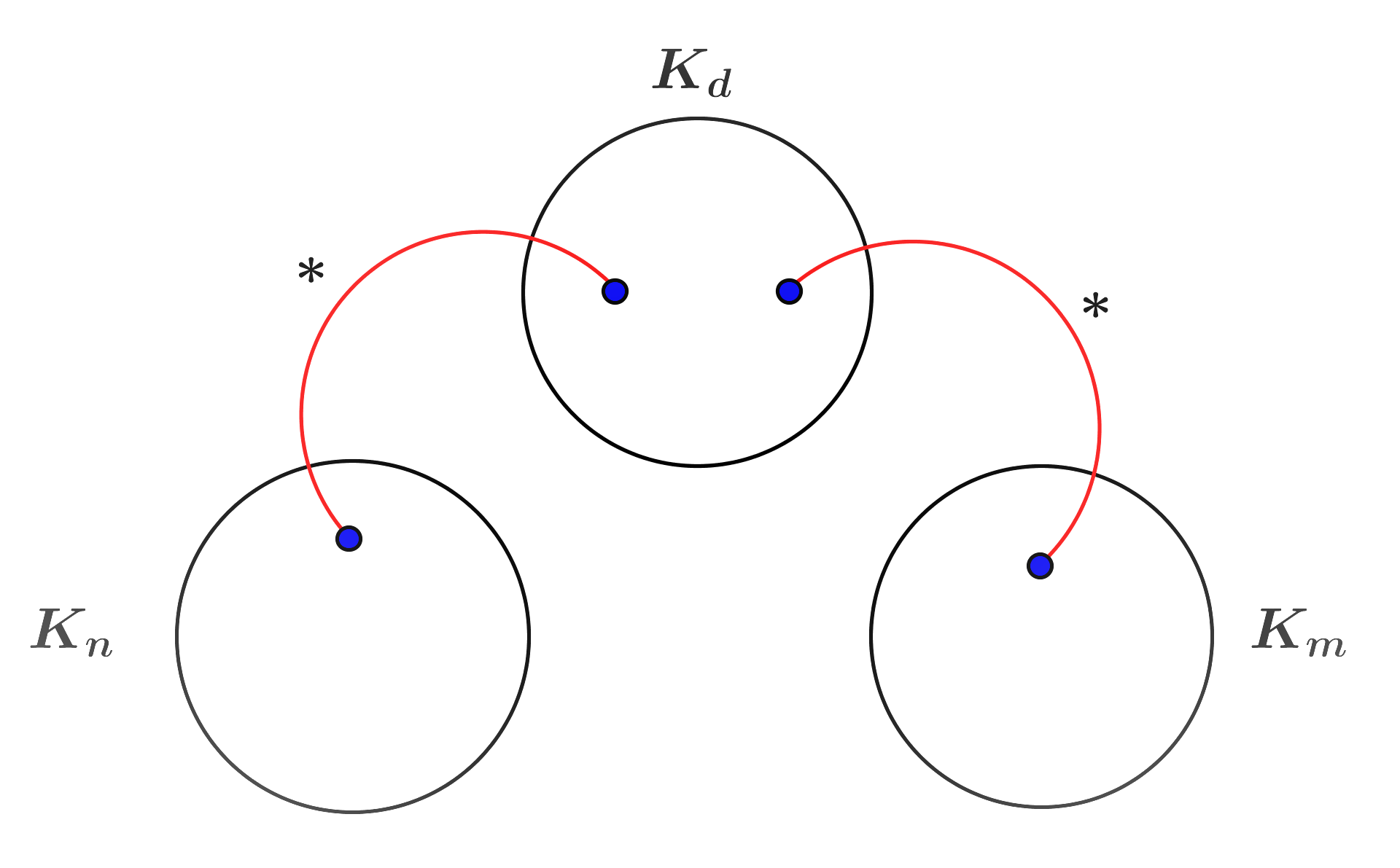}
	\end{center}
	\caption{\footnotesize  Quasi-$(\frac{n}{m},2)$-DB graph $ K_{n} *  K_{d} *  K_{m} $ with $ n>m $.  \normalsize}
\end{figure}\normalsize

In Figure 8, the graphs $ K_{n}, K_{d}$ and $ K_{m} $ with $ n>m $, are joint together and has a common edge created by the black nodes. 
This graph can be also represented by $ G=K_{n-2} *  K_{2} * K_{d-4} * K_{2} *  K_{m-2} $. 
Moreover, for any $ x\in  K_{n} $ and $ y\in  K_{m} $ with $ d(x,y)=3 $ we get 
$ \vert W_{x\underline{3}y}^{G}\vert =n $ and $ \vert W_{y\underline{3}x}^{G}\vert =m. $

\begin{figure}[H]\label{pic8}\footnotesize 
	\begin{center}
		\hskip 0.5 cm\includegraphics[scale=0.15]{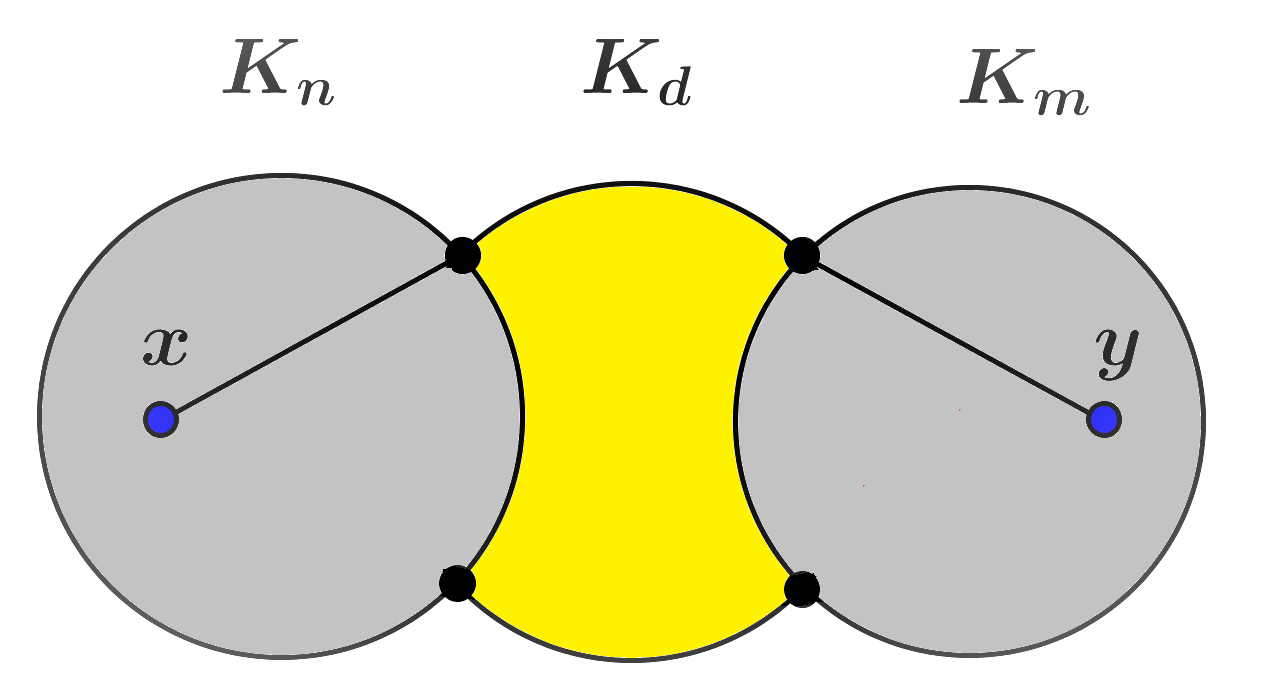}
	\end{center}
	\caption{\footnotesize  Quasi-$(\frac{n}{m},3)$-DB graph $ K_{n-2} *  K_{2} * K_{d-4} * K_{2} *  K_{m-2} $ with $ n>m $.  \normalsize}
\end{figure}\normalsize

  Figure 9 shows another quasi-$(\frac{n}{m},3)$-DB graph which is a cycle-shaped graph. 

\begin{figure}[H]\label{pic9}\footnotesize 
	\begin{center}
		\hskip 0.5 cm\includegraphics[scale=0.10]{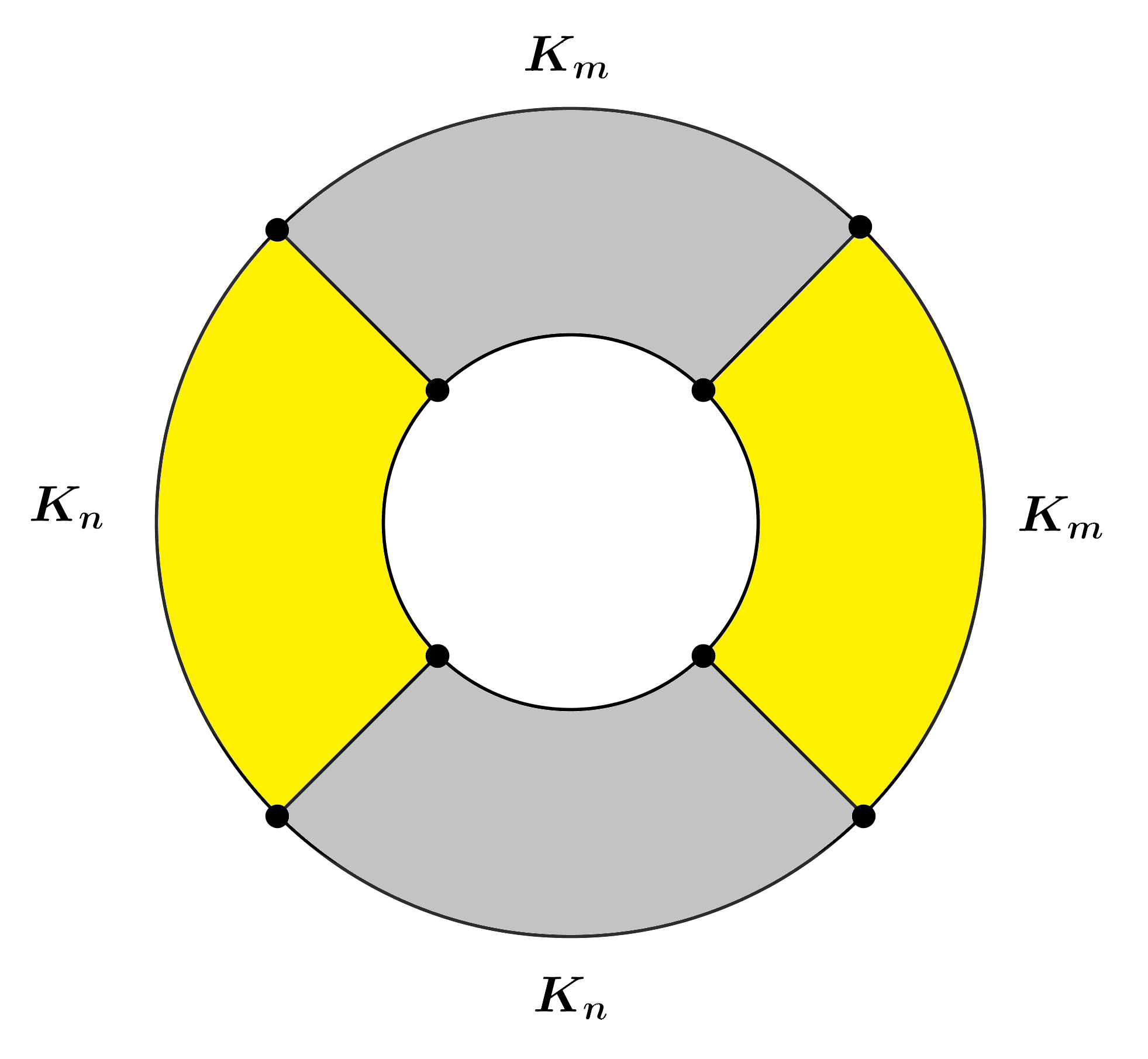}
	\end{center}
	\caption{\footnotesize  Quasi-$(\frac{n}{m},3)$-DB graph 
		$ \overline{K_{m-4} *  K_{2} * K_{m-4} * K_{2} *  K_{n-4} * K_{2} *  K_{n-4} * K_{2}} $
	  with $ n>m $.  \normalsize}
\end{figure}\normalsize


In Figure 10, for any $ x\in  K_{n} $ and $ y\in  K_{m} $ with $ d(x,y)=4 $ we have
$$ \vert W_{x\underline{4}y}^{G}\vert =2m+n-8, \quad      \vert W_{y\underline{4}x}^{G}\vert =2n+m-8, $$
where 
$$ G=\overline{K_{m-4} *  K_{2} * K_{n-4} * K_{2} *  K_{m-4} * K_{2} *  K_{n-4} * K_{2}
*  K_{m-4} * K_{2} **  K_{n-4} * K_{2}}, \quad n>m. $$

\begin{figure}[H]\label{pic10}\footnotesize 
	\begin{center}
		\hskip 0.5 cm\includegraphics[scale=0.10]{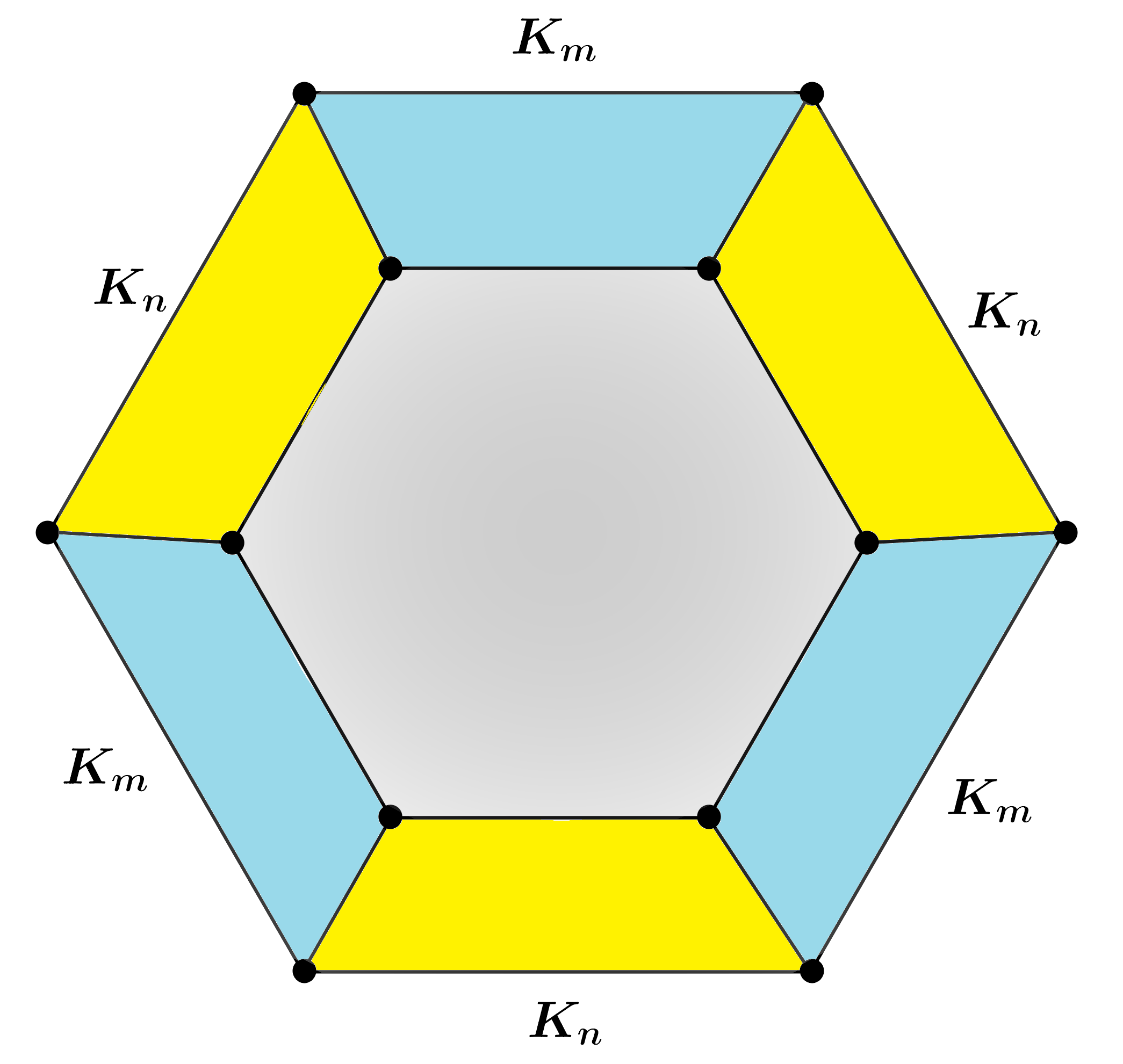}
	\end{center}
	\caption{\footnotesize  Quasi-$(\frac{2n+m-8}{2m+n-8},4)$-DB graph $ G $ for $ n>m $.  \normalsize}
\end{figure}\normalsize
 
  Inspired by the Figure 10, we easily create 
 the quasi-$(\lambda,2k)$-DB graph by joining the subgraphs $K_{n} $ and $  K_{m}$, alternately, where 
 $$ \lambda= \frac{kn+(k-1)m-4k}{km+(k-1)n-4k}. $$
 \\
 
Depicted by Figure 11, the graph $ G $ includes four subgraphs $ K_{p} $ which are connected by the operation $ * $ and 
 for any $ x\in  K_{n} $ and $ y\in  K_{m} $ with $ d(x,y)=5 $ we have
 $$ \vert W_{x\underline{5}y}^{G}\vert =2p+n-4, \quad      \vert W_{y\underline{5}x}^{G}\vert =2p+m-4. $$
 
 \begin{figure}[H]\label{pic11}\footnotesize 
 	\begin{center}
 		\hskip 0.5 cm\includegraphics[scale=0.30]{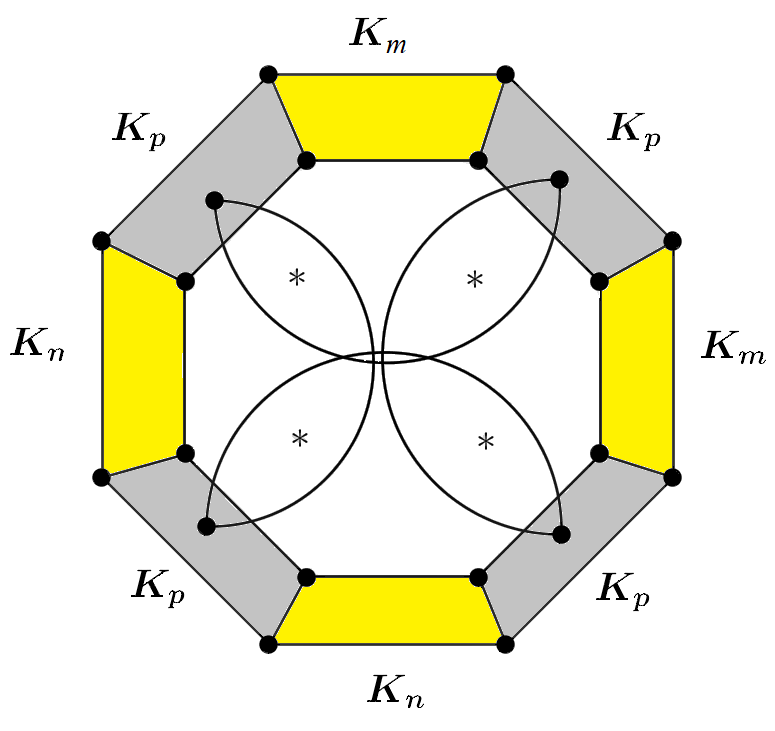}
 	\end{center}
 	\caption{\footnotesize  Quasi-$(\frac{2p+n-4}{2p+m-4},5)$-DB graph $ G $ for $ n>m $.  \normalsize}
 \end{figure}\normalsize
 
 In Figure 12, using four local groups of subgraphs $ K_{p} $ including $ k-1 $ graphs $ K_{p} $ connected by the operation 
 $ * $ together with $K_{n} $ and $  K_{m}$ we have  a quasi-$( \lambda,2k+1)$-DB graph $ G $ 
 for 
 $$\lambda =\frac{2(k-1)n+m-4(k-1)}{2(k-1)m+n-4(k-1)}. $$

 \begin{figure}[H]\label{pic12}\footnotesize 
	\begin{center}
		\hskip 0.5 cm\includegraphics[scale=0.12]{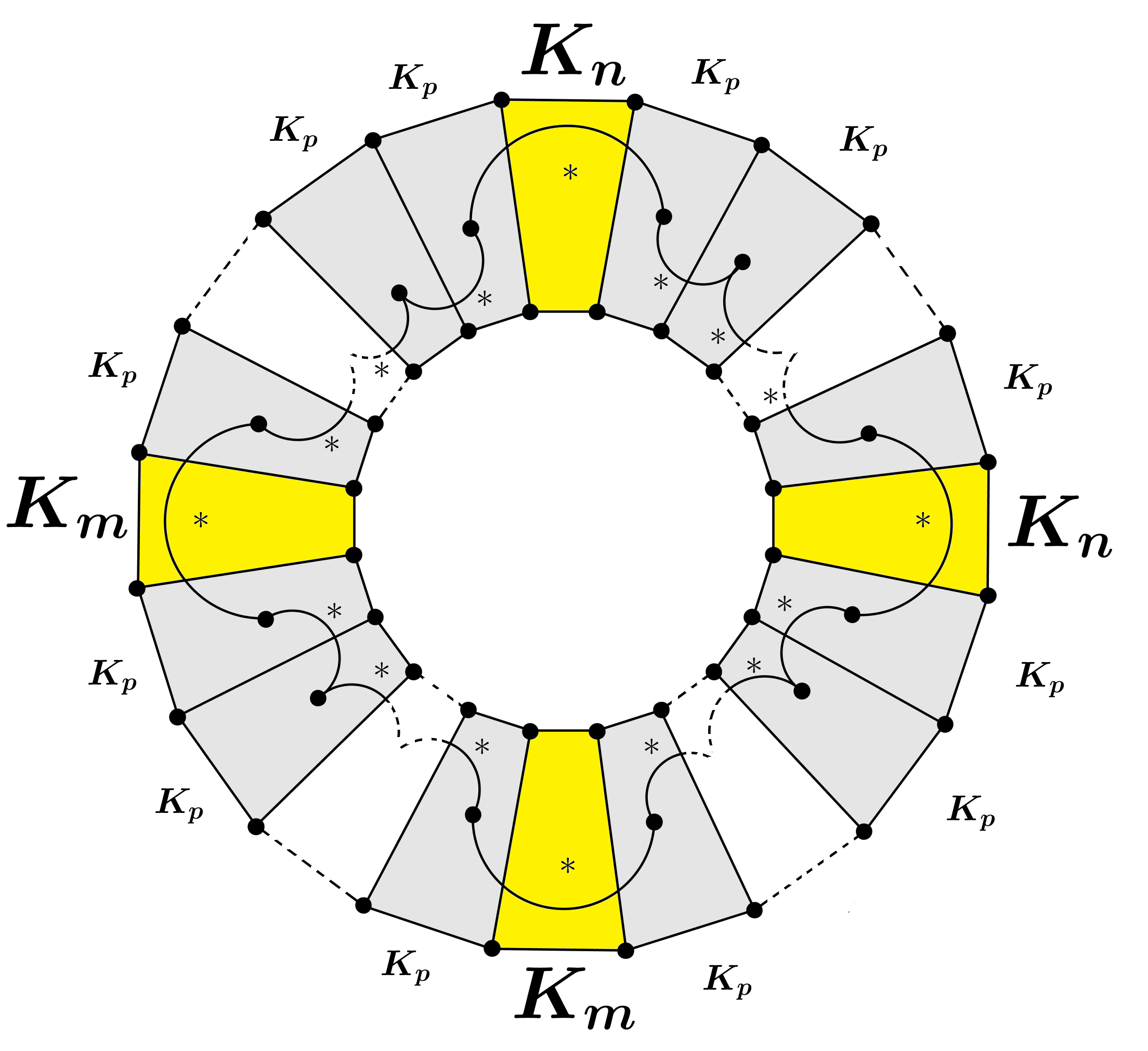}
	\end{center}
	\caption{\footnotesize  Quasi-$(\frac{2(k-1)n+m-4(k-1)}{2(k-1)m+n-4(k-1)},2k+1)$-DB graph $ G $ for $ n>m $.  \normalsize}
\end{figure}\normalsize
 
 \begin{proposition}
 	Let $ G $ be a connected quasi-$(\lambda,2)$-DB graph. If $ \delta(G) = 1 $, then $ G $ is isomorphic to 
 	a complete graph with some pendant vertices with no same root. 
 \end{proposition}
 
 \begin{proof} 
 	For the connected quasi-$(\lambda,2)$-DB graph  $ G $ with $ \delta(G) = 1 $ suppose that 
 	$ P:xyz $ is a path in graph $ G $ with $ \text{deg}(x)=1 $, $ d(x,z)=2 $ and  $ x,y,z\in V(G). $ 
 	It is clear that 
 	$ \vert W_{x\underline{2}z}^{G}\vert =1, \  \lambda=\vert W_{z\underline{2}x}^{G}\vert= |V(G) |-2. $ 
 	Now, for any pair of vertices $ u,v $ in $ G $ with $ d(u,v)=2 $ we have 
 	$$   \vert W_{u\underline{2}v}^{G}\vert= \lambda \vert W_{v\underline{2}u}^{G}\vert 
 	\quad \text{or} \quad \vert W_{v\underline{2}u}^{G}\vert= \lambda \vert W_{u\underline{2}v}^{G}\vert.  $$
 	It means that $ \vert W_{v\underline{2}u}^{G}\vert =1 $ or $ \vert W_{u\underline{2}v}^{G}\vert =1, $ 
 	that is, $ u $ or $ v $ is pendant vertex but not both. (Note that if $ \text{deg}(u)= \text{deg}(v)=1   $ 
 	then they both have the same root, that is, they are adjacent to a unique vertex in $ G  $ and it 
 	contradicts to the fact that  $ G $ is quasi-$(\lambda,2)$-DB.)
 	This also shows that it only remains a graph with diameter 1  if we remove the pendant vertices of $ G $. 
 	Therefore, $ G $ is a complete graph with some pendant vertices with no same root (see Figure 13).
 \end{proof}

 \begin{figure}[H]\label{pic13}\footnotesize 
	\begin{center}
		\hskip 0.5 cm\includegraphics[scale=0.25 ]{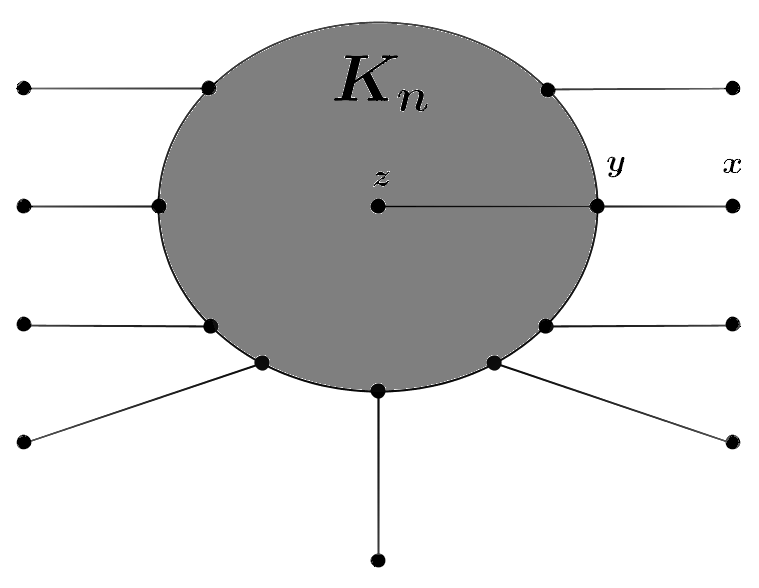}
	\end{center}
	\caption{\footnotesize  Graph $ G $ formed by a complete graph with some pendant vertices attached.  \normalsize}
\end{figure}


\normalsize
 
 \begin{thebibliography}{00}
	
	
\bibitem{R1}  Abedi, A., Alaeiyan, M., Hujdurovi\'{c}, A. ,  Kutnar, K., Quasi-$\lambda $-distance-balanced graphs, 
	\textit{Discrete Appl. Math.} \textbf{227} (2017) 21--28.


\bibitem{R5}  Balakrishnan, K.,  Changat, M., Peterin, I., \v{S}pacapan, S.,  \v{S}parl, P., Subhamathi, A.R., Strongly distance-balanced graphs and graph products, \textit{European J. Combin.} 30 (2009) 1048--1053.


\bibitem{R3} Faghani, M., Pourhadi, E., Further remarks on $ n $-distance-balanced graphs, \textit{J. Indones. Math. Soc.}, (2019), 
In press. 

\bibitem{R4}  Faghani, M., Pourhadi, E., $ n $-edge-distance-balanced graphs, \textit{Italian J. Pure Appl. Math.} 38, (2017) 18--31.

\bibitem{s} Faghani, M., Pourhadi, E. and Kharazi, H., On the new extension of distance-balanced graphs, {\em Trans. Combin.} \textbf{5} (4)  (2016) 21--34.  

\bibitem{p} Hammack, R., Imrich, W. and Klav\u{z}ar, S., \textit{Handbook of product graphs}, CRC Press, Taylor \& Francis Group, 2011. 

\bibitem{R2} Hujdurovi\'{c}, A. , On some properties of quasi-distance-balanced graphs, \textit{Bull. Aust. Math. Soc.} \textbf{97} (2018), 177--184.

\bibitem{h} Ili\'{c}, A., Klav\u{z}ar, S. and Milanovi\'{c},  M., On distance-balanced graphs, {\em European J. Combin.} \textbf{31} (2010) 733--737.

\bibitem{R6} W. Imrich, S. Klav\v{z}ar, Product graphs: structure and recognition,  John Wiley, 2000.

\bibitem{a} Jerebic, J., Klav\u{z}ar, S. and Rall, D.F., Distance-balanced graphs, {\em Ann. Combin.} \textbf{12:1}   (2008) 71--79.
\bibitem{c} Khalifeh, M.H.,Yousefi-Azari, H., Ashrafi, A.R. and Wagner, S.G., Some new results on distance-based graph invariants, {\em European J. Combin.} \textbf{30} (2009) 1149--1163.
\bibitem{d} Kutnar, K., Malni\u{c}, A., Maru\u{s}i\u{c}, D. and Miklavi\u{c}, \u{S}., Distance-balanced graphs: Symmetry conditions, {\em Discrete Math.} \textbf{306} (2006) 1881--1894.
\bibitem{g} Kutnar, K., Malni\u{c}, A., Maru\u{s}i\u{c}, D. and Miklavi\u{c}, \u{S}., The strongly distance-balanced property of the generalized Petersen graphs, {\em Ars Math. Contemp.} \textbf{2} (2009) 41--47.
\bibitem{m} Kutnar, K. and Miklavi\u{c}, \u{S}., Nicely distance-balanced graphs, {\em European J. Combin.} \textbf{39} (2014) 57--67.
\bibitem{f} Miklavi\u{c}, \u{S}. and \u{S}parl, P., On the connectivity of bipartite distance-balanced graphs, {\em European J. Combin.} \textbf{33} (2012) 237--247.
\bibitem{n} Tavakoli, M., Yousefi-Azari, H. and Ashrafi, A.R., Note on edge distance-balanced graphs,  {\em Trans. Combin.} \textbf{1:1}  (2012) 1--6.
	
	
\end{thebibliography}
\end{document}